\documentclass[10pt, reqno]{amsart}
\numberwithin{equation}{section}
\usepackage{amssymb, color, cite}
\usepackage{hyperref}

\newcommand{\qtq}[1]{\quad\text{#1}\quad}

\newcommand{\R}{\mathbb{R}}
\newcommand{\C}{\mathbb{C}}

\newcommand{\eps}{\varepsilon}

\newcommand{\F}{\mathcal{F}}

\newtheorem{theorem}{Theorem}[section]

\newtheorem{proposition}[theorem]{Proposition}

\theoremstyle{definition}

\theoremstyle{remark}

\begin{document}
\title[Averaging]{A note on averaging for \\ the dispersion-managed NLS}
\author{Jason Murphy}
\address{Department of Mathematics, University of Oregon}
\email{jamu@uoregon.edu}

\begin{abstract} We discuss averaging for dispersion-managed nonlinear Schr\"odinger equations in the fast dispersion management regime, with an application to the problem of constructing soliton-like solutions to dispersion-managed nonlinear Schr\"odinger equations. 
\end{abstract}

\maketitle

\section{Introduction}

Our interests in this note are in averaging phenomena and soliton-type solutions for dispersion-managed nonlinear Schr\"odinger equations. Here `dispersion-managed' refers to the presence of a time-periodic factor in the linear part of the equation.  These equations arise in the setting of nonlinear optics, e.g. in the setting of laser light propagating down a fiber optics cable in which the dispersion varies periodically. The basic idea is that by varying the dispersion periodically in such a way that the average dispersion is small, one can suppress the undesired effects of dispersion on signal propagation (e.g. pulse broadening).  In particular, dispersion management is meant to have a stabilizing effect on pulse propagation. See e.g. \cite{ZGJT}, as well as \cite{TuritsynBF} for an extensive review.

To fix ideas, we will restrict our attention to the focusing cubic equation in $3d$ with positive average dispersion, i.e. 
\begin{equation}\label{nls}
i\partial_t u + \gamma(t)\Delta u = - |u|^2 u,\quad (t,x)\in\R\times\R^3,
\end{equation}
with $\gamma:\R\to\R$ a $1$-periodic function satisfying
\[
\int_0^1 \gamma(t)\,dt =1. 
\]
In the setting of nonlinear optics, a typical example is that of a piecewise constant function $\gamma$ that varies periodically.

There has been significant recent mathematical interest in dispersion-managed nonlinear Schr\"odinger equations.  For the physical background of such equations, as well as well-posedness results and related topics, we refer the reader to \cite{AntonelliSautSparber, CHoL, CHL, CL, EHL, GT1, GT2, GMO, HL, HL2, MV, MVH, PZ, ZGJT, Agrawal, Kurtzke} (and remark that this list is far from exhaustive). 

To the best of the author's knowledge, the question of the existence of solitons for the dispersion-managed equation in the form \eqref{nls} still seems to be open.  On the other hand, there are a wealth of results concerning solitary wave solutions for closely-related dispersion-managed models (e.g. averaged models or other approximate models). Many results involve the study of solitons for a related averaged equation arising in the \emph{strong dispersion management regime}.  In this setting, one considers \eqref{nls} with dispersion maps of the form $\tfrac{1}{\eps}\gamma(\tfrac{t}{\eps})$ and takes the limit as $\eps\to 0$.  In this case, one arrives at a limiting equation in which the time dependence is removed from the linear part of the equation and the nonlinearity is replaced with a nonlocal version, namely
\[
i\partial_t u = -\langle \gamma\rangle\Delta u -\int_0^1 e^{-i\tau\Delta}\bigl[|e^{i\tau\Delta}u|^2 e^{i\tau\Delta}u\bigr]\,d\tau,\quad \langle\gamma\rangle=\int_0^1 \gamma(t)\,dt. 
\]
For more details, the reader may consult references such as \cite{CL, Lushnikov, GT1, EHL, HL, MV, PZ, ZGJT, TuritsynBF, HL2, CLL, CHL}.  In particular, we refer the reader to \cite{TuritsynBF} for an extensive review article on the topic of dispersion-managed solitons. For results concerning the case of \emph{zero} average dispersion (which we will not consider here), one can refer to works such as \cite{AntonelliSautSparber, HL2}. We would also like to highlight the work of Pelinovsky \cite{Pel}, which studies the problem using a Gaussian ansatz, as an example of a particularly interesting result in this area.

It is also possible to obtain the standard cubic NLS (i.e. \eqref{nls} with $\gamma\equiv 1$) as an averaged version of \eqref{nls} by considering the so-called \emph{fast dispersion management regime}.  This entails considering the solutions to the equations
\begin{equation}\label{nls-eps}
i\partial_t u + \gamma(\tfrac{t}{\eps})\Delta u = -|u|^2 u
\end{equation}
and taking the limit as $\eps\to 0$ (see e.g. \cite{AntonelliSautSparber, BK1, BK2, YKT, CMVH}).  In this case, the existence of solitons for the averaged equation is well-known (with the specific combination of positive average dispersion and focusing nonlinearity).  In particular, proving convergence for solutions to \eqref{nls-eps} as $\eps\to 0$ could provide an approach to constructing soliton-like solutions to dispersion-managed nonlinear Schr\"odinger equations.  This is the basic idea considered in this paper.

In \cite{CMVH}, we proved some averaging results for the cubic equation in $2d$, which is an $L^2$-critical problem.  Essentially, we proved that we can obtain convergence on any time interval on which the solution to the underlying equation obeys suitable space-time bounds.  In particular, in the defocusing case, one can obtain global-in-time averaging due to the result of \cite{Dodson}.  In fact, we proved two results in \cite{CMVH}.  The first was a subcritical result (inspired by the paper \cite{AntonelliSautSparber}), treating initial data belonging to $H^s$ for some $s>0$.  In this scenario, we could spend a bit of regularity to obtain quantitative (in $\eps$) estimates for the difference of the linear propagators associated to \eqref{nls} and \eqref{nls-eps}. The second result treated $L^2$ data (for a more restrictive class of dispersion maps), utilizing a change of variables from \cite{Fanelli} and adapting techniques from the work \cite{Ntekoume} (on spatial homogenization for the $2d$ cubic NLS).

As soliton solutions do not obey global space-time bounds, the techniques presented in \cite{CMVH} only yield convergence for \eqref{nls-eps} on fixed time intervals.  In this note, we adapt the techniques of \cite{CMVH} (specifically, those used for the subcritical result) to the $3d$ cubic equation and slightly refine the argument in order to obtain convergence on a longer (although still finite) time interval in the case of soliton data.  Before stating our main result, we introduce some notation and terminology. 

First, we denote by $Q$ the ground state soliton for the cubic NLS, i.e. the unique radial, nonnegative, and decaying solution to the equation
\[
-Q+\Delta Q = -Q^3
\]
(see e.g. \cite{Weinstein}).

Next, we introduce the notion of an admissible dispersion map (cf. \cite{MVH}). We call $\gamma:\R\to\R$ \emph{admissible} if $\gamma$ is $1$-periodic, $\gamma$ and $\tfrac{1}{\gamma}$ both belong to  $L^\infty$, and $\gamma$ has at most finitely many discontinuities in $[0,1]$.  We will discuss local well-posedness for \eqref{nls} and \eqref{nls-eps} with admissible dispersion maps in Section~\ref{S:2} below. 

Finally, given a time interval $I\subset\R$, we introduce the Strichartz spaces $S^s(I)$ via the norm
\[
\|u\|_{S^{s}(I)}=\| u\|_{L_t^\infty H_x^{s}(I\times\R^3)} + \|u\|_{L_t^{\frac{10}{3}} H_x^{s,\frac{10}{3}}(I\times\R^3)}. 
\]

Our main result is the following:
\begin{theorem}\label{T} Let $\gamma$ be an admissible dispersion map with $\int_0^1 \gamma\,dt = 1$. Given $\eps>0$, let $u^\eps$ denote the solution to \eqref{nls-eps} with $u^\eps|_{t=0}=Q$. 

There exists $a,b>0$ such that for $\eps>0$ sufficiently small, the solution $u^{\eps}$ exists on $I_\eps:=[-\log(\eps^{-a}),\log(\eps^{-a})]$ and obeys
\[
\|u^\eps(t) - e^{it}Q\|_{S^{\frac12}(I_\eps)}\lesssim \eps^b. 
\]
\end{theorem}

The strategy of proof is based on the prior work \cite{CMVH}, which in turn built on ideas from \cite{AntonelliSautSparber}. In particular, we let $u^\eps$ and $u$ denote the solutions to \eqref{nls-eps} and \eqref{nls}, respectively, both with initial data given by $Q$, the standard NLS ground state. Using the Duhamel formula, we decompose the difference
\[
u^\eps(t)-u(t) 
\]
into two types of terms.  In the first type of term, we can exhibit at least one copy of $u^\eps-u$.  In particular, these terms can be incorporated into a bootstrap estimate on sufficiently small intervals.  In the second type of term, we can exhibit the difference of propagators 
\begin{equation}\label{diff}
e^{i\Gamma_\eps(t,s)\Delta}-e^{i(t-s)\Delta}
\end{equation}
(see \eqref{D:Gamma} for the definition of $\Gamma_\eps$). We rely on quantitative Strichartz estimates for \eqref{diff} (as in \cite{CMVH, AntonelliSautSparber}) to prove that such terms are $\mathcal{O}(\eps^c)$ for some $c>0$ (see Theorem~\ref{Strichartz} below). In particular, we can iterate over $\approx |\log \eps|$ small intervals and thereby obtain Theorem~\ref{T}.

Theorem~\ref{T} demonstrates the existence of soliton-like solutions to dispersion-managed nonlinear Schr\"odinger equations on long time intervals.  We note, however, that the techniques presented here do not seem likely to establish any type of global-in-time result. Indeed, the basic \emph{a priori} estimate that plays a key role in the proof of Theorem~\ref{T} (see \eqref{apriori} below) is only useful on small time intervals.  Refinements of this approach will be considered in future work.

We also remark that the result as presented here does not depend on the fact that $Q$ is the ground state soliton.  For example, the techniques here may be applied to establish averaging results based around any traveling wave solution to \eqref{nls}.

\subsection*{Acknowledgements} The author was supported in part by NSF grant DMS-2350225.
\section{Preliminaries}\label{S:2}

We use the standard notation $A\lesssim B$ to denote $A\leq CB$ for some $C>0$.  We make regular use of the Strichartz norms
\begin{equation}\label{S-norm}
\|u\|_{S^{s}(I)}=\| u\|_{L_t^\infty H_x^{s}(I\times\R^3)} + \|u\|_{L_t^{\frac{10}{3}} H_x^{s,\frac{10}{3}}(I\times\R^3)}. 
\end{equation}
Here we use the notation
\[
\|u\|_{H_x^{s,r}(\R^3)} = \|u\|_{L_x^r(\R^3)} + \| |\nabla|^s u\|_{L_x^r(\R^3)}. 
\]
The fractional derivative $|\nabla|^s$ is defined as a Fourier multiplier operator: $|\nabla|^s = \F^{-1}|\xi|^s \F$.  We also use the notation $\langle \nabla\rangle^s = \F^{-1}(1+|\xi|^2)^{s/2}\F.$ 

Throughout this section, we fix an admissible dispersion map $\gamma$ satisfying 
\[
\int_0^1 \gamma(t) \,dt = 1.
\]
Here \emph{admissibility} is defined as in \cite{MVH}: specifically, we require that $\gamma$ is $1$-periodic, that $\gamma$ and $\tfrac{1}{\gamma}$ both belong to  $L^\infty$, and that $\gamma$ has at most finitely many discontinuities in $[0,1]$.

Given $\eps>0$, the solution to the linear equation 
\begin{equation}\label{linear-eps}
\begin{cases}
&i\partial_t u + \gamma(\tfrac{t}{\eps})\Delta u=0, \\
&u(t,t_0)= \varphi
\end{cases}
\end{equation}
is given by
\[
u(t)= e^{i\Gamma_\eps(t,t_0)\Delta}\varphi,
\]
where
\begin{equation}\label{D:Gamma}
\Gamma_\eps(t,t_0):=\int_{t_0}^t \gamma(\tfrac{\tau}{\eps})\,d\tau. 
\end{equation}

We will need some estimates from \cite{MVH, CMVH} (see also \cite{AntonelliSautSparber}).  The first estimate yields Strichartz estimates for \eqref{linear-eps} that hold uniformly in $\eps$.  The latter two estimates establish convergence of the propagators $e^{i\Gamma_\eps(t,t_0)\Delta}$ to $e^{i(t-t_0)\Delta}$ as $\eps\to 0$, which in turn relies on the basic but essential fact that 
\begin{equation}\label{Gamma_est}
|\Gamma_\eps(t,t_0)-(t-t_0)|\lesssim\eps
\end{equation}
(see \cite[Lemma~2.1]{CMVH}).  In the present paper, we have specialized to the case that $\langle\gamma\rangle:=\int_0^1\gamma = 1$.

We call $(q,r)$ a Schr\"odinger admissible pair (in three space dimensions) if $2<q\leq\infty$ and $\tfrac{2}{q}+\frac{3}{r}=\tfrac{3}{2}$.  We omit the $L_t^2$ endpoint due to the use of the Christ--Kiselev lemma in \cite{CMVH}.

\begin{theorem}[Strichartz estimates; convergence of propagators\cite{MVH, CMVH}]\label{Strichartz}

Given an admissible pair $(q,r)$, we have the uniform Strichartz estimate
\begin{equation}\label{uniform}
\|e^{i\Gamma_\eps(t,t_0)\Delta}\|_{L^2\to L_t^q L_x^r}\lesssim_\gamma 1 \qtq{uniformly in}\eps>0.
\end{equation}
Furthermore, 
\begin{equation}\label{cop1}
\|e^{i\Gamma_\eps(\cdot,t_0)\Delta}-e^{i(\cdot-t_0)\Delta}\|_{H^\theta\to L_t^q L_x^r} \lesssim_\gamma \eps^{(1-\frac{2}{q})\frac{\theta}{2}},
\end{equation}
and if $(\tilde q,\tilde r)$ is any other Schr\"odinger admissible pair,
\[
\biggl\|\int_{t_0}^t [e^{i\Gamma_\eps(t,s)\Delta}-e^{i(t-s)\Delta}]F(s)\,ds\biggr\|_{L_t^q L_x^r} \lesssim_\gamma \eps^{(2-\frac{2}{q}-\frac{2}{\tilde q})\frac{\theta}{2}}\|\langle \nabla\rangle^{2\theta}F\|_{L_t^{\tilde q'}L_x^{\tilde r'}},
\]
where $'$ denotes the H\"older dual. 
\end{theorem}

We next record a local well-posedness result for \eqref{nls-eps}. We construct solutions to the Duhamel formula
\begin{equation}\label{duh}
u(t)=e^{i\Gamma_\eps(t,0)\Delta}\varphi+i\int_0^t e^{i\Gamma_\eps(t,s)\Delta}|u|^2 u(s)\,ds. 
\end{equation}
We remark that in the following proposition, the interval of existence depends on the initial condition but \emph{not} on $\eps$.  This stems from the fact that the Strichartz estimates appearing in \eqref{uniform} hold uniformly in $\eps$.

\begin{proposition}[Local well-posedness]\label{P:LWP} Fix $s\in[\frac35,1]$ and an admissible dispersion map $\gamma$. Let $\varphi\in H^s(\R^3)$ and $\eps>0$.  Then there exists $T=T(\|\varphi\|_{H^s})$ and a solution $u^\eps:(-T,T)\times\R^3\to\C$ to \eqref{nls-eps} with $u^\eps|_{t=0}=\varphi$.  The solution belongs to $S^s((-T,T))$ and may be extended as long as its $S^{\frac12}$-norm remains finite.
\end{proposition}

\begin{proof}  Recall that the implicit constants in the Strichartz estimates in \eqref{uniform} are uniform in $\eps>0$. We will focus on showing existence forward in time only.  

The proof follows from the usual contraction mapping argument, i.e. showing that the map 
\[
u\mapsto\Phi(u):=\text{RHS}\eqref{duh}
\]
is a contraction on a suitable complete metric space. 

We let $T>0$ to be determined below and take our metric space to be
\[
X=\{u:\|u\|_{S^s([0,T])}\leq C\|\varphi\|_{H^s}\},
\]
where $C$ is related to the implicit constants in Strichartz estimates and Sobolev embedding, with distance given by
\[
d(u,v) = \|u-v\|_{L_{t,x}^{\frac{10}{3}}([0,T]\times\R^3)}.
\]

To see that $\Phi:X\to X$, we let $u\in X$ and apply Strichartz estimates.  Focusing on the contribution of the inhomogeneous terms, we use the fractional chain rule, H\"older's inequality, and Sobolev embedding to estimate
\begin{align*}
\||u|^2 u\|_{L_t^1 H_x^s} & \lesssim T^{\frac{1}{10}}\|u\|_{L_t^{\frac{10}{3}}L_x^{10}}^2 \|u\|_{L_t^{\frac{10}{3}}H_x^{s,\frac{10}{3}}} \\
& \lesssim T^{\frac{1}{10}} \| |\nabla|^{\frac35} u\|_{L_t^{\frac{10}{3}}L_x^{\frac{10}{3}}}^2 \|u\|_{L_t^{\frac{10}{3}}H_x^{s,\frac{10}{3}}}  \\
&\lesssim T^{\frac{1}{10}}[C\|\varphi\|_{H^s}]^3 \leq \tfrac12 C \|\varphi\|_{H^s}
\end{align*}
for $T=T(\|\varphi\|_{H^s})$ sufficiently small. Thus we may obtain that $\Phi:X\to X$. 

Choosing $u,v\in X$ and estimating similarly, we can obtain
\begin{align*}
d(\Phi(u),\Phi(v)) & \lesssim \bigl\{\|u\|_{L_t^{\frac{10}{3}} \dot H_x^{\frac35,\frac{10}{3}}}^2 + \|v\|_{L_t^{\frac{10}{3}}\dot H_x^{\frac35,\frac{10}{3}}}^2\bigr\}\|u-v\|_{L_{t,x}^{\frac{10}{3}}} \\
& \lesssim T^{\frac{1}{10}}[C\|\varphi\|_{H^s}]^2 \|u-v\|_{L_{t,x}^{\frac{10}{3}}} \\
& \leq \tfrac12 d(u,v)
\end{align*}
for $T=T(\|\varphi\|_{H^s})$ sufficiently small. 

It follows that $\Phi$ is a contraction on $X$ and hence has a unique fixed point, yielding our desired solution.

It remains to show that the solution may be continued as long as its $S^{\frac12}$-norm remains finite.  To see this, first note that the local existence result just proven guarantees that a solution may be extended as long as its $H^s$-norm remains finite; that is, if the solution cannot be extended past some time $T_*>0$, we must have that $\|u(t)\|_{H^s}\to\infty$ as $t\uparrow T^*$.  

Thus it suffices to prove that if a solution exists on some interval $I$ and satisfies
\begin{equation}\label{LWPif}
\|u\|_{S^{\frac12}(I)}<\infty,
\end{equation}
then 
\begin{equation}\label{LWPthen}
\|u\|_{S^s(I)}<\infty.
\end{equation}

Suppose \eqref{LWPif} holds.  We let $\eta>0$ to be determined below and split $I$ into finitely many intervals $I_j=[t_j,t_{j+1}]$ such that
\[
\|u\|_{L_t^{\frac{10}{3}} H_x^{\frac12,\frac{10}{3}}(I_j\times \R^3)}<\eta. 
\]
Restricting to an interval of the form $[t_j,t]$, we can estimate essentially as we did above to obtain 
\begin{align*}
\|u\|_{S^s([t_j,t])}&  \lesssim \|u(t_j)\|_{H^s} + \||u|^2 u\|_{L_t^{\frac{10}{9}}H_x^{s,\frac{30}{17}}} \\
& \lesssim \|u(t_j)\|_{H^s} + \|u\|_{L_t^{\frac{10}{3}}L_x^{\frac{15}{2}}}^2 \|u\|_{L_t^{\frac{10}{3}} H_x^{s,\frac{10}{3}}} \\
& \lesssim \|u(t_j)\|_{H^s} + \| u\|_{L_t^{\frac{10}{3}} \dot H_x^{\frac12,\frac{10}{3}}}^2 \|u\|_{S^s([t_j,t])} \\
& \lesssim \|u(t_j)\|_{H^s} + \eta^2 \|u\|_{S^s([t_j,t])}.
\end{align*}
Thus, by a standard continuity argument, we may obtain that 
\[
\|u\|_{S^s(I_j)} \leq 2C\|u(t_j)\|_{H^s}.
\]
Iterating over the finite collection of intervals, we can obtain the desired conclusion \eqref{LWPthen}.\end{proof}

We remark once again that the time of existence in Proposition~\ref{P:LWP} depends on the initial data, but \emph{not} on $\eps$.  This is a consequence of the fact that the Strichartz estimates for $e^{i\Gamma_\eps(t,t_0)\Delta}$ are uniform in $\eps$. 

Note also that we have not optimized the preceding result in terms of the regularity of the data.  Indeed, the argument could be extended to any subcritical regularity (i.e. data in $H^s$ for $s>\tfrac12$).  

On the other hand, obtaining a critical result (i.e. working with data in $H^{\frac12}$) that is uniform in $\eps>0$ is a bit more subtle.  In this case, one would like to choose the existence time $T>0$ small enough that
\[
\|e^{i\Gamma_\eps(t,0)\Delta}\varphi\|_{L_t^{\frac{10}{3}} H_x^{\frac12,\frac{10}{3}}([0,T]\times\R^3)} \ll 1.
\]
For fixed $\eps>0$, this is indeed possible by the monotone convergence theorem, using the fact that
\[
\|e^{i\Gamma_\eps(t,0)\Delta}\varphi\|_{L_t^{\frac{10}{3}} H_x^{\frac12,\frac{10}{3}}(\R\times\R^3)} \lesssim \|\varphi\|_{H^{\frac12}}.
\]
However, even though the implicit constant in this Strichartz estimate is uniform in $\eps>0$, it is not clear that one can choose $T$ independent of $\eps>0$.  

One way to proceed is to require a bit of extra regularity on the data $\varphi$ and utilize \eqref{cop1}. In particular, if we take $\varphi\in H^{\frac12+\theta}$ and suppose (without loss of generality) that $\int_0^1 \gamma=1$, then for any $T>0$ we can estimate
\begin{align*}
\|e^{i\Gamma_\eps(t,0)\Delta}& \varphi\|_{L_t^{\frac{10}{3}} H_x^{\frac12,\frac{10}{3}}([0,T]\times\R^3)} \\
& \lesssim \|e^{it\Delta}\varphi\|_{L_t^{\frac{10}{3}}H_x^{\frac12,\frac{10}{3}}([0,T]\times\R^3)} + \eps^{\frac{\theta}{5}}\|\varphi\|_{H^{\frac12+\theta}}.
\end{align*} 
By choosing $T=T(\varphi)$ and $\eps>0$ sufficiently small, we can make this quantity arbitrarily small.  In this way, one could obtain a local theory using only `critical spaces' (albeit for data slightly more regular than $H^{\frac12}$) that holds uniformly for (small) $\eps>0$.

\section{Proof of the main result}

We turn to the proof of Theorem~\ref{T}.  We focus on proving estimates forward in time only.

\begin{proof}[Proof of Theorem~\ref{T}] We let $\gamma$ be an admissible dispersion map with 
\[
\int_0^1 \gamma(t)\,dt = 1,
\]
and define $\gamma_\eps$ and $\Gamma_\eps$ as in the previous section.

Given $\eps>0$, we apply Proposition~\ref{P:LWP} and let $u^\eps$ be the solution to \eqref{nls-eps} with $u^\eps|_{t=0}=Q$. We also let $u(t)=e^{it}Q$, which solves 
\[
i\partial_t u + \Delta u = -|u|^2 u
\]
and exists globally in time.  Our goal is to estimate the difference between the solutions $u^\eps$ and $u$.

By Proposition~\ref{P:LWP}, the solutions $u^\eps$ exist on some interval $[0,T]$ (for some $T$ independent of $\eps$) and may be continued as long as their $S^{\frac12}$-norms remain under control (recall the definition of this norm in \eqref{S-norm}). Thus, in what follows we will assume that the solutions $u^\eps$ exist and establish \emph{a priori} bounds on the difference between $u^\eps$ and $u$ in the $S^{\frac12}$-norm.  The implicit constants below will generally depend on the fixed dispersion map $\gamma$, but \emph{not} on $\eps$. 

Fix $t,t_0\in\R$ and denote $F(z)=|z|^2 z$. We begin by using the Duhamel formula to write
\begin{equation}\label{decomp}
\begin{aligned}
u^\eps(t)-u(t) & = e^{i\Gamma_\eps(t,t_0)\Delta}[u^\eps(t_0)-u(t_0)] \\
& \quad + [e^{i\Gamma_\eps(t,t_0)\Delta}-e^{i(t-t_0)\Delta}]u(t_0) \\
& \quad + i\int_{t_0}^t e^{i\Gamma_\eps(t,s)\Delta}[F(u^\eps(s))-F(u(s))]\,ds \\
& \quad + i\int_{t_0}^t [e^{i\Gamma_\eps(t,s)\Delta}-e^{i(t-s)\Delta}]F(u(s))\,ds. 
\end{aligned}
\end{equation}
Letting $I\ni t_0$, we can therefore use the estimates from the proof of Proposition~\ref{P:LWP} and Theorem~\ref{Strichartz} and obtain
\begin{align*}
\|&u^\eps-u\|_{S^{\frac12}} & \\
&\lesssim \|u^\eps(t_0)-u(t_0)\|_{H^{\frac12}}+\eps^c \|u\|_{L_t^\infty H_x^{s}} \\
& \quad + \|F(u^\eps)-F(u)\|_{L_t^{\frac{10}{9}}H_x^{\frac12,\frac{30}{17}}}+ \eps^c \|F(u)\|_{L_t^{\frac{10}{9}}H_x^{s,\frac{30}{17}}}
\end{align*}
for some $s\in(\tfrac12,1)$ and $c=c(s)>0$, where all norms are taken over $I\times\R^3$. 

We first observe that 
\[
\eps^c \|u\|_{L_t^\infty H_x^s} \lesssim_Q \eps^c.
\]

We next observe that $F(u^\eps)-F(u)$ may be written as a sum of terms of the form
\[
vw[u^\eps-u], \quad v,w\in\{u,u^\eps-u\}
\]
up to complex conjugation.  Thus, applying the fractional product rule, using the same spaces as in the proof of Proposition~\ref{P:LWP}, applying Young's inequality, and recalling that $u(t)=e^{it}Q$ with $Q$ smooth and rapidly decaying, we may obtain
\begin{align*}
\|F(u^\eps)-F(u)\|_{L_t^{\frac{10}{9}}H_x^{\frac12,\frac{30}{17}}} &\lesssim \|u\|_{L_t^{\frac{10}{3}}H_x^{\frac12,\frac{10}{3}}}^2\|u^\eps-u\|_{S^\frac12} + \|u^\eps-u\|_{S^{\frac12}}^3 \\
& \lesssim_Q |I|^{\frac35}\|u^\eps-u\|_{S^{\frac12}} + \|u^\eps-u\|_{S^{\frac12}}^3. 
\end{align*}

Finally, we have
\[
\eps^c \|F(u)\|_{L_t^{\frac{10}{9}} H_x^{s,\frac{30}{17}}} \lesssim_Q \eps^c |I|^{\frac{9}{10}}. 
\]

Combining the estimates above, we arrive at our basic \emph{a priori} estimate:
\begin{equation}\label{apriori}
\begin{aligned}
\|u^\eps-u\|_{S^{\frac12}(I)}& \leq C\bigl[ \|u^\eps(t_0)-u(t_0)\|_{H^{\frac12}}+\eps^c   +|I|^{\frac35}\|u^\eps-u\|_{S^{\frac12}(I)} \\
& \quad \quad + \|u^\eps-u\|_{S^{\frac12}(I)}^3 + \eps^c |I|^{\frac{9}{10}}\bigr].
\end{aligned}
\end{equation}

We will now proceed by using this estimate iteratively to propagate good bounds for $u^\eps-u$ over sufficiently small intervals.  To this end, we define 
\begin{equation}\label{D:Teps}
T_\eps=\log[\eps^{-a}]
\end{equation}
 for some $a\in(0,c)$. We split $[0,T_\eps]$ into $J\sim T_\eps$ intervals of the form 
 \[
 I_j=[t_j,t_{j+1}]
 \]
 so that $|I_j|\ll1$ for each $j$. 
 
 On any interval $I=[t_j,t]\subset I_j$ the \emph{a priori} estimate implies 
\begin{equation}\label{apriori}
\|u^\eps-u\|_{S^{\frac12}(I)} \leq 2C \big[\|u^\eps(t_0)-u(t_0)\|_{H^{\frac12}}+ 2\eps^c + \|u^\eps-u\|_{S^{\frac12}(I)}^3\bigr]. 
\end{equation}

We now define 
\[
A_0 = 8C \qtq{and} A_j = 4C[A_{j-1}+2] \qtq{for}1\leq j\leq J.
\]
One can verify by induction that 
\[
A_j\leq (16C)^{J+1} \qtq{for all}j,
\]
and in particular (recalling the definition of $T_\eps$) we have 
\begin{equation}\label{Ajbd}
A_j\eps^c\ll1
\end{equation}
for all $j$, provided $\eps$ is sufficiently small.

We will prove by induction that
\begin{equation}\label{induction}
\|u^\eps-u\|_{S(I_j)} \leq A_j \eps^c \qtq{for all}j.
\end{equation}

For $j=0$, we use the fact that $u^\eps(t_0)=u(t_0)=Q$, so that \eqref{apriori} implies 
\[
\|u^\eps-u\|_{S^{\frac12}(I)} \leq 4C\eps^c + 2C\|u^\eps - u\|_{S^{\frac12}(I)}^3
\]
for any $I=[0,t]\subset I_0$.  By a continuity argument, this implies
\[
\|u^\eps-u\|_{S^{\frac12}(I_0)}\leq 8C\eps^c
\]
provided $\eps$ is sufficiently small. This yields the base case. 

Now suppose that \eqref{induction} holds up to level $j-1$.  Then \eqref{apriori} implies
\begin{align*}
\|u^\eps-u\|_{S^{\frac12}(I)} & \leq 2CA_{j-1} \eps^c + 4C\eps^c + 2C\|u^\eps - u\|_{S^{\frac12}(I)}^3 \\
& \leq \tfrac12A_j\eps^c + 2C\|u^\eps-u\|_{S^{\frac12}(I)}^3. 
\end{align*}
for $I=[t_j,t]\subset I_j$.  As $A_j\eps^c\ll 1$ (cf. \eqref{Ajbd}), another continuity argument implies
\[
\|u^\eps-u\|_{S^{\frac12}(I_j)}\leq A_j\eps^c,
\]
thus completing the induction. Recalling the definition of $T_\eps$ in \eqref{D:Teps}, we can now obtain 
\[
\|u^\eps-u\|_{S^{\frac12}([0,T_\eps])}\lesssim A_J \eps^c \lesssim \eps^b
\]
for some $b>0$, thus completing the proof of Theorem~\ref{T}. 
\end{proof}

\end{document}